\documentclass[a4paper,11pt,oneside,reqno]{amsart}
\usepackage[T1]{fontenc}
\usepackage[dvipdf]{xcolor}
\usepackage{bbm, dsfont}
\usepackage{amsmath,amsthm,amsfonts,amssymb}
\usepackage[left=3cm,right=3cm, top=3.5cm, bottom=3.5cm]{geometry}
\usepackage[dvips]{graphicx}
\usepackage[english]{babel}
\usepackage{enumerate}

\newcommand{\bu}{\mathbf{u}}
\newcommand{\bua}{\mathbf{u}^\alpha}

\newcommand{\buna}{\mathbf{u}_n^\alpha}
\newcommand{\bv}{\mathbf{v}}
\newcommand{\bva}{\bv^\alpha}
\newcommand{\bvan}{\bva_n}
\newcommand{\bw}{\mathbf{w}}
\newcommand{\by}{y^\alpha_n}
\newcommand{\el}{\mathbb{L}}
\newcommand{\ve}{\mathbf{V}}
\newcommand{\h}{\mathbf{H}}

\newcommand{\bh}{\mathbb{H}}
\newcommand{\Lve}{\lVert}
\newcommand{\Rve}{\rVert}
\newcommand{\me}{\mathbb{E}}

\newcommand{\sA}{\mathrm{A}}
\newcommand{\hsa }{\sA^\frac12}
\newcommand{\ipsa}{(\Id+\alpha^2 \sA)^{-1}}

\newcommand{\na}{\mathrm{N}_\alpha}
\newcommand\toup{\nearrow}%+\int_0^t \langle \bvan(s), \sA\na P_n Q dW(s)\rangle

\newcommand{\RR}{\mathbb{R}}
\newcommand{\MO}{\mathcal{O}}
\newcommand{\rK}{\mathrm{K}}
\def\eps{\varepsilon}
\numberwithin{equation}{section}

\DeclareMathOperator{\lsp}{linspan}
\DeclareMathOperator{\Id}{I}

\theoremstyle{plain}
\newtheorem{condition}{Condition}[section]
\newtheorem{assum}[condition]{Assumption}

\newtheorem{lem}{Lemma}[section]
\newtheorem{thm}[lem]{Theorem}
\newtheorem{prop}[lem]{Proposition}
\newtheorem{cor}[lem]{Corollary}
\theoremstyle{definition}
\newtheorem{Def}[lem]{Definition}
\newtheorem{Rem}[lem]{Remark}
\title[Rate of convergence of the 2-D stochastic Leray-$\alpha$ to  the 2-D stochastic NSE]{On the rate of convergence of the 2-D stochastic Leray-$\alpha$ model
to the 2-D stochastic Navier-Stokes equations with multiplicative noise}

\author[H. Bessaih]{Hakima Bessaih}
\address[H. Bessaih]{Department of Mathematics, University of Wyoming, 1000 East University Avenue, Laramie WY 82071, United States,}
\author[P. Razafimandimby]{Paul Razafimandimby}
\address[P.~Razafimandimby]{Department of Mathematics and Information Technology, Montanuniversit\"at Leoben,
Fr. Josefstr. 18, 8700 Leoben, Austria}
\email[H.~Bessaih]{bessaih@uwyo.edu}
\email[P.~Razafimandimby]{paul.razafimandimby@unileoben.ac.at}
\begin{document}
\begin{abstract}
 In the present paper we study the convergence of the solution of  the two dimensional (2-D)  stochastic Leray-$\alpha$ model to the solution of the 2-D stochastic Navier-Stokes equations. 
 We are mainly interested in the rate, as $\alpha\to 0$, of the following error function 
 \begin{equation*}
 \eps_\alpha(t)=\sup_{s\in [0,t]} \lvert \bua(s)-\bu(s)\rvert+\left( \int_0^t \lvert \hsa [\bua(s)-\bu(s)] \rvert^2 ds \right)^\frac12,
\end{equation*}
  where $\bua$ and $\bu$ are the solution of stochastic Leray-$\alpha$ model and the stochastic Navier-Stokes equations, respectively.
  We show that when properly localized the error function $\eps_\alpha$ converges in mean square as $\alpha\to 0$ and the convergence is of order $O(\alpha)$. We also prove
   that $\eps_\alpha$ converges in probability to zero with order at most $O(\alpha)$.
\end{abstract}
\maketitle
\section{Introduction}
%%%%%%%%%%%%%%%%%%%%%%%%%%%%%%%%%%%%%%%%%%%%%%%%%%%%%%%%%%%%%
%%%%%%%%%%%%%%%%%%%%%%%%%%%%%%%%%%%%%%%%%%%%%%%%%%%%%%%%%%%%%
%%%%%%%%%%%%%%%%%%%%%%%%%%%%%%%%%%%%%%%%%%%%%%%%%%%%%%%%%%%%%
%%%%%%%%%%%%%%%%%%%%%%%%%%%%%%%%%%%%%%%%%%%%%%%%%%%%%%%%%%%%%
%%%%%%%%%%%%%%%%%%%%%%%%%%%%%%%%%%%%%%%%%%%%%%%%%%%%%%%%%%%%%
%%%%%%%%%%%%%%%%%%%%%%%%%%%%%%%%%%%%%%%%%%%%%%%%%%%%%%%%%%%%%
The Navier-Stokes system is  the most used model in turbulence theory. 
In recent years, various regularization models were introduced as an efficient subgrid model scale of the Navier-Stokes equations, see for eg. \cite{CLT, Chen-1, Chen-2, Chen-3, CHOT, Foias1, Foias2, HT, ILT, LL}. Moreover, numerical analyses in 
\cite{CHMZ, GH, HT, LL-0, LKT, LKTT, MKSM, NS} seem to confirm that these models  can capture remarkably well the physical phenomenon of turbulence in fluid flows at a lower
computational cost. Among them are the Navier-Stokes-$\alpha$, Leray-$\alpha$, modified Leray-$\alpha$, Clark-$\alpha$ to name just a few.

Another tool used to tackle the closure problem in turbulent flows is to introduce a stochastic forcing that will mimic all the terms that can't be handled. This approach is basically motivated by
Reynolds' work which stipulates that hydrodynamic turbulence is
composed of slow (deterministic) and fast (stochastic) components. This approach was used in \cite{ROZOVSKII1} to derive a stochastic Navier-Stokes equations
with gradient and nonlinear diffusion coefficient. 

 It is worth emphasizing that the presence of the stochastic term
 (noise) in the model often leads to qualitatively new types of
 behavior, which are very helpful in understanding real
 processes and is also often more realistic. In particular, for the 2d Navier-Stokes equations, 
 some ergodic properties are proved when adding a random perturbation, 
 %see \cite{Flandoli2, KAP, FF+BM, MH+JM}. 

There is an extensive literature about the convergence of $\alpha$-models  to the Navier-Stokes equations, see for eg.  \cite{Foias2, BBF, Caraballo1, Caraballo2, IC+SK, 
DEUGOUE3, DEUGOUE2, DEUGOUE4, DRS}.
However, only a few papers deal with the rate of convergence, see \cite {Titi, waymire}. 
In \cite{Titi} the rates of convergence of
four $\alpha$-models (NS-$\alpha$ model, Leray-$\alpha$ model, modified Leray-$\alpha$ model, and
simplified Bardina model) in the two-dimensional (2D) case, subject to
periodic boundary conditions on the periodic box $[0, L]^2$ are studied. The authors of \cite{Titi} mainly showed that all the four $\alpha$-models 
have the same order of convergence and error estimates; that is, the convergences in the $\el^2$-norms  are all of the order $O\left(\frac\alpha L \big(\log(\frac L \alpha) \big)^\frac12 \right)$ as 
$ \frac\alpha L$ tends to zero, while in \cite{waymire} the rate of convergence  of order 
$O(\alpha)$ is obtained in a mixed $L^1-L^2$ time-space norm with small initial data in Besov-type function spaces.

Despite the numerous papers, there are 
only very few addressing the convergence of stochastic $\alpha$-models to the stochastic Navier-Stokes. It is proved in \cite{IC+SK} that the stochastic Leray-$\alpha$ 
model has a unique invariant measure which converges to the stationary solution (unique invariant measure) of 3-D (resp. 2-D) stochastic Navier-Stokes equations. In \cite{DEUGOUE3} and 
\cite{DEUGOUE2} Deugou\'e and Sango proved that one can find a sequence of weak martingales of the 3-D stochastic Navier-Stokes-$\alpha$ and Leray-$\alpha$ model respectively which converges in 
distribution to the weak martingale solution of the 3-D stochastic Navier-Stokes equations.
%%%%%%%%%%%%%%%%%%%%%%%%%%%%%%%%%%%%%%%%%%%%%%%%%%%%%%%%%%%%%%%%%%%%%%%%%%%%%%%%%%
%%%%%%%%%%%%%%%%%%%%%%%%%%%%%%%%%%%%%%%%%%%%%%%%%%%%%%%%%%%%%%%%%%%%%%%%%%%%%%%%%%
%%%%%%%%%%%%%%%%%%%%%%%%%%%%%%%%%%%%%%%%%%%%%%%%%%%%%%%%%%%%%%%%%%%%%%%%%%%%%%%%%%
%%%%%%%%%%%%%%%%%%%%%%%%%%%%%%%%%%%%%%%%%%%%%%%%%%%%%%%%%%%%%%%%%%%%%%%%%%%%%%%%%%

Here in this paper,  we are interested in the analysis of the rate of convergence of the two-dimensional
stochastic Leray-$\alpha$ model to the stochastic Navier-Stokes Equations. 
More precisely we consider the Leray-$\alpha$ model with multiplicative stochastic perturbation on a periodic domain $\MO=[0,L]^2$, $L>0$, given by the following  system 
\begin{subequations}\label{LAM}
\begin{align}
 d\bv^\alpha(t) +[\nu\sA\bv^\alpha(t) +B(\bu^\alpha(t), \bv^\alpha(t))]dt= Q{(\bu^\alpha(t))}dW(t), \, t\in (0,T] \label{LAM-1}\\
 \bu^\alpha+\alpha^2 \sA\bu^\alpha=\bv^\alpha,\label{LAM-2}\\
 \bu^\alpha(0)=\bu_0,
\end{align}
\end{subequations}
where $W$ is a cylindrical Wiener process on a separable Hilbert space $\rK$, $A$ is the Stokes operator and $B$ is the well-known bilinear map in the mathematical theory of the 
Navier-Stokes equations. We refer to Section \ref{notation} for the functional setting. 
 
Our main goal in the present paper is to 
study the convergenve of the solution $\bua$ to \eqref{LAM} to the solution of the stochastic Navier-Stokes equations given by 
\begin{subequations}\label{NSE}
\begin{align}
 d\bu(t)=[-\sA\bu(t)-B(\bu(t),\bu(t))]dt +Q(\bu(t)) dW(t), \, t\in (0, T],\\
 \bu(0)=\bu_0.
\end{align}
\end{subequations}
To the best of our knowledge it seems that the investigation of the rate 
of convergence of the stochastic $\alpha$-model to the stochastic Navier-Stokes has never been done before.
In this paper we initiate this direction of research by studying the rate of convergence of the error function
\begin{equation*}
 \eps_\alpha(t)=\sup_{s\in [0,t]} \lvert \bua(s)-\bu(s)\rvert+\left( \int_0^t \lvert \hsa [\bua(s)-\bu(s)] \rvert^2 ds \right)^\frac12,
\end{equation*}
  as $\alpha$ tend to zero. Here $\lvert \cdot \rvert$ denotes the $\el^2(\MO)$-norm. By deriving several important uniform estimates for the sequence of stochastic processes $\bua$ we can prove that for an 
  appropriate family of stopping times $\{\tau_R; R>0\}$ the stopped error function
  $\eps_\alpha(t\wedge \tau_R)$ converges to $0$ in mean square as $\alpha$ goes to zero and the convergence is of order $O(\alpha)$. In particular, this shows that when 
  the error function $\eps_\alpha$ is properly localized then the order of convergence in the stochastic case is better than the one in the deterministic case. 
  In this paper, we also 
   prove that the convergence in probability (see for example \cite{Prin} for the definition) of $\eps_\alpha$ is also of order $O(\alpha)$. These results can be found in 
    Theorem \ref{ORDER} and Theorem \ref{ORDER-PR}. We mainly combine the approaches used in \cite{BBM} and in \cite{Titi}.
   
   In section \ref{notation} we introduce the notations and some  frequently used lemmata.
   In Section  \ref{AP-EST-SEC},  we introduce the main assumptions on the diffusion coefficient.
   Moreover, several important uniform estimates which are the
backbone of our analysis will be derived. In Section \ref{RATE},  we state and prove our main results; we mainly show that when properly
localized the error function $\eps_\alpha$ converges in mean square to zero as $\alpha$ tend to zero. Owing to the uniform estimates obtained in Section 
\ref{AP-EST-SEC},  we also show in Section  \ref{RATE} that 
it converges in probability with order $O(\alpha)$ as $\alpha$ tends to zero. 

Throughout the paper $C$, $c$ denote some unessential constants which do not depend on $\alpha$ and may change from one place to the next one. 
   
\section{Notations}\label{notation}
%  Let $n\in \{2,3\}$ and assume that $\MO \subset \mathbb{R}^n$ is a Poincar\'e's domain
%  (its definition is given below) with
%  boundary $\partial \MO$ of class $\mathcal{C}^\infty$.
In this section we introduce some notations that are frequently used in this paper. We will mainly follow the presentation of
 Cao and Titi \cite{Titi}.

 Let $\MO$ be a bounded subset of $\RR^2$. For any $p\in [1,\infty)$ and $k\in \mathbb{N}$,  $\el^p(\MO)$ and
$\mathbb{W}^{k,p}(\MO)$ are the well-known Lebesgue and Sobolev
spaces, respectively, of $\mathbb{R}^2$-valued functions. The
corresponding spaces of scalar functions we will denote by
standard letter, e.g. ${W}^{k,p}(\MO)$. The usual scalar product on $\el^2$ is denoted by $\langle
u,v\rangle$ for $u,v\in \el^2$. Its associated norm is $\lvert
u\rvert$, $u\in \el^2$.

Let $L>0$ and $\mathcal{P}$ be the set of (periodic) trigonometric polynomials of two variables
defined on the periodic
domain $\MO=[0,L]^2$ and with zero spatial average; that is, for every $\phi\in \mathcal{P}$, $\int_\MO \phi(x) dx=0$.
We also set
\begin{align*}
\mathcal{V}& =\left\{ \bu\in [\mathcal{P}]^2\,\,\text{such that}%
\,\,\nabla \cdot \bu=0\right\} \\
\mathbf{V}& =\,\,\text{closure of $\mathcal{V}$ in }\,\,\mathbb{H}^{1}(\MO) \\
\mathbf{H}& =\,\,\text{closure of $\mathcal{V}$ in
}\,\,\mathbb{L}^{2}(\MO).
\end{align*}
 We endow the spaces $\h$ with the scalar product and norm of $\el^2$. We equip the space $\ve$
with the scalar product
$(( \bu, \bv)):=\int_\MO \nabla \bu(x) \cdot \nabla \bv(x) dx $ which is equivalent to the
$\bh^1(\MO)$-scalar product on $\ve$.  The norm corresponding to the scalar product $((\cdot, \cdot ))$ is denoted by
$\Lve \cdot \Rve$.

Let $\Pi: \el^2 \rightarrow \h$ be the projection
from $\el^2$ onto $\h$. We denote by $\sA$ the
Stokes operator defined by
\begin{equation}\label{STOKES}
\begin{cases}
 D(\sA)=\{ u\in \h, \; \Delta u \in \h\},\\
 \sA u =-\Pi\Delta u, \; u\in D(\sA).
\end{cases}
\end{equation}
Note that in the space-periodic case
$$ \sA \bu=-\Pi \Delta \bu=-\Delta \bu, \text{ for all } \bu \in D(\sA).$$
The operator $\sA$ is a self-adjoint, positive definite, and a compact operator on $\h$ (see, for instance, \cite{PC+CF-88, Temam}).
We will denote by $\lambda_1\le \lambda_2\le \ldots$ the eigenvalues of $\sA$; the correspoding eigenfunctions
$\{\Psi_i: i=1, 2, \ldots\}$ form an orthonormal basis of $\h$ and an
orthogonal basis of $\ve$.
For any positive integer $n\in \mathbb{N}$ we set $$ \h_n =\lsp\{\Psi_i: i=1,\ldots, n\}$$ and we denote by
$P_n$ the orthogonal projection onto $\h_n$ defined by
$$ P_n \bu =\sum_{i=1}^n \langle \bu, \Psi_i\rangle \Psi_i, \text{  for all } \bu \in \h.$$
We also recall that in the periodic case we have $D(\sA^\frac n2)=\bh^n(\MO)\cap \h$, for $n>0$ (see, for instance, \cite{PC+CF-88, Temam}).
In particular we have $\ve=D(\sA^\frac12)$.

For every $\bw\in \ve$, we have the following Poincar\'e inequality
\begin{equation}\label{Poincare-1}
 \lambda_1 \lvert \bw \rvert^2 \le  \lVert \bw \rVert^2, \text{ for all } \bw \in \ve.
\end{equation}
Also, there exists $c>0$ such that
\begin{equation}\label{Poincare-2}
 c\lvert \sA \bw \rvert \le \Lve \bw \Rve_2 \le c^{-1} \lvert \sA \bw \rvert \text{  for every } \bw \in D(\sA),
\end{equation}
\begin{equation}\label{Poincare-3}
 c \lvert \sA^\frac12 \bw \rvert \le \Lve \bw \Rve_1 \le c^{-1} \lvert \sA^\frac12 \bw \rvert \text{  for every }
 \bw \in \ve.
\end{equation}
Thanks to \eqref{Poincare-3} the norm $\Lve \cdot \Rve$ of $\ve$ is equivalent to the usual $\bh^1(\MO)$-norm.
Recall that the following estimate, valid for all $\bw \in \bh^1$ (or $\bw \in H^1$), is a special case of
Gagliardo-Nirenberg's inequalities:
\begin{align}
\Lve \bw\Rve_{\el^4}\le c \lvert \bw\rvert^{\frac12} \lvert \nabla \bw
\rvert^\frac12. \label{GAG-l4}
% \\
% \Lve \bu \Rve_{\el^\infty}\le \Lve \bu \Rve_{\el^4}^{1-a}\Lve
% \nabla \bu \Rve_{\el^4}^a.\label{GAG-LInf}
\end{align}
The inequality \eqref{GAG-l4} can be written in the spirit of
the continuous embedding
\begin{equation}\label{SOB-EM}
\bh^1\subset \el^4.
\end{equation}

Next, for avery $\bw_1,\bw_2\in \mathcal{V}$ we define the bilinear operator
\begin{equation}\label{BIL-B1}
 B(\bw_1,\bw_2)=\Pi[(\bw_1\cdot\nabla)\bw_2].
\end{equation}
In the following lemma we recall some properties of the bilinear operator $B$.
\begin{lem}
 The bilinear operator $B$ defined in \eqref{BIL-B1} satisfies the following
 \begin{enumerate}[(i)]
  \item $B$ can be extended as a continuous bilinear map $B:\ve \times \ve \to \ve^\ast$, where $\ve^\ast$
  is the dual space of $\ve$. In particular, the following inequalities/equalities hold for all $\bu,\bv, \bw \in \ve$:
  \begin{align}
   \lvert \langle B(\bu,\bv),\bw\rangle \rvert \le &  c \vert \bu \vert^\frac12 \Lve \bu \Rve^\frac12 \Lve \bv \Rve
   \vert \bw \vert^\frac12 \Lve \bw \Rve^\frac12, \label{BIL-B2}\\
   \langle B(\bu,\bv),\bw\rangle=& -\langle B(\bu,\bw), \bv\rangle. \label{BIL-B3}
  \end{align}
As consequence of \eqref{BIL-B3} we have
\begin{equation}\label{BIL-B4}
 \langle B(\bu,\bv),\bv\rangle=0
\end{equation}
for all $\bu, \bv\in \ve$.
\item In the 2D periodic boundary condition case, we have
\begin{equation}\label{BIL-B5}
 \langle B(\bu,\bu),\sA\bu\rangle=0,
\end{equation}
for every $\bu \in D(\sA)$.
 \end{enumerate}
\end{lem}
\begin{proof}
 The proof of the above lemma can be found, for instance, in \cite{PC+CF-88, Temam}.
\end{proof}
We also recall the following lemma.
\begin{lem}
 For every $\bu \in D(\sA)$ and $\bv \in \ve$, we have
 \begin{equation}
  \lvert \langle B(\bv,\bu), \sA\bu\rangle \rvert \le c \Lve \bv \Rve \Lve \bu \Rve \,\lvert \sA \bu\rvert.
 \end{equation}
\end{lem}
\begin{proof}
 For the proof we refer to \cite{Titi}.
\end{proof}
\section{A priori estimates for the stochastic Navier-Stokes equations and the stochastic Leray-$\alpha$ model}\label{AP-EST-SEC}
The stochastic Leray-$\alpha$ model \eqref{LAM} and the stochastic
Navier-Stokes equations \eqref{NSE} have been extensively studied.
Their well posedness are established in several mathematical
papers. In this section we just recall the most recent results
which are very close to our purpose. Most of these results were
obtained from Galerkin approximation and energy estimates.
However, the estimates derived in previous papers are not
sufficient for our analysis. Therefore, we will also devote this
section to derive several important estimates which are the
backbone of our analysis.

%Before we proceed to our analysis we should introduce the main hypotheses on the coefficient of the %noise. We mainly assume that $Q$ satisfies the following set of conditions.
We consider a prescribed complete probability system $(\Omega, \mathcal{F}, \mathbb{P})$ equipped with a filtration $\mathbb{F}:=\{\mathcal{F}_t; t \ge0\}$. We assume that the filtration satisfies the usual 
condition, that is, the family $\mathbb{F}$ is increasing, right-continuous and $\mathcal{F}_0$ contains all null sets of $\mathcal{F}$.
Let $\rK$ be a separable Hilbert space. On the filtered probability space $(\Omega, \mathcal{F}, \mathbb{F}, \mathbb{P})$ we suppose that we are given a cylindrical Wiener process $W$ on $\rK$.

For two Banach spaces $X$ and $Y$, we denote by $\mathcal{L}(X,Y)$ the space of all bounded linear maps $L: X\rightarrow Y$. The space of all Hilbert-Schmidt operators $L:X\rightarrow Y$ is denoted 
by $\mathcal{L}_2(X,Y)$. The Hilbert-Schmidt norm of $L\in \mathcal{L}_2(X,Y)$ is denoted by $\Vert L\Vert_{\mathcal{L}_2(X,Y)}$. When $X=Y$ we just write  $\mathcal{L}_2(X):=\mathcal{L}_2(X,X)$.

Now, we can introduce the standing assumptions of the paper.
\begin{assum}\label{HYP-Q}
 Throughout this paper we assume that $Q:D(\sA^\frac12 ) \rightarrow \mathcal{L}(\rK,D(\hsa)) $ satisfies:
 \begin{enumerate}[(i)]
  \item \label{HYP-Q-i} there exists $\ell_0>0$ such that for any $\bu_1, \bu_2\in D(\hsa)$ we have
  \begin{equation*}
   \Lve Q(\bu_1)-Q(\bu_2)\Rve_{\mathcal{L}_2(\rK,\h)}\le \ell_0 \lvert \bu_1-\bu_2\rvert,
  \end{equation*}
  \item \label{HYP-Q-ii} there exists $\ell_1>0$ such that for any $\bu_1, \bu_2\in D(\hsa)$ we have
  \begin{equation*}
   \Lve \hsa \big[Q(\bu_1)-Q(\bu_2)\big]\Rve_{\mathcal{L}_2(\rK,\h)}\le \ell_1 \lvert \hsa \bu_1-\hsa \bu_2\rvert.
  \end{equation*}
 \end{enumerate}
\end{assum}
\begin{Rem}\label{Rem-HYP}
 Assumption \ref{HYP-Q} implies in particular that
 \begin{enumerate}[(1)]
  \item \label{Rem-HYP-i} there exists $\ell_2>0$ such that  for any $\bu \in D(\hsa)$ we have
  \begin{equation*}
   \Lve Q(\bu)\Rve_{\mathcal{L}_2(\rK,\h)}\le \ell_2 (1+ \lvert \bu\rvert),
  \end{equation*}
  \item \label{Rem-HYP-ii} there exists $\ell_3>0$ such that  for any $\bu \in D(\hsa)$ we have
  \begin{equation*}
   \Lve \hsa Q(\bu)\Rve_{\mathcal{L}_2(\rK,\h)}\le \ell_3 (1+ \lvert \hsa \bu\rvert),
  \end{equation*}
 \end{enumerate}
\end{Rem}

\subsection{A priori estimates for the Navier-Stokes equations}\label{AP-EST-NS}

The study of the stochastic Navier-Stokes equations was pioneered by Bensoussan and Temam 
in \cite{Bensoussan-Temam}. Since then, an intense investigation about the qualitative and quantitative  properties of this model has generated an extensive literature, see e.g.   
\cite{Bensoussan, BCF, ZB+SP,IC+AM-1,Flandoli_Gatarek_95, ROZOVSKII1}.

The following definition of solution is mainly taken from \cite{IC+AM-1} (see also \cite{Bensoussan,Flandoli_Gatarek_95}).
\begin{Def}
 A weak solution to \eqref{NSE} is a stochastic process $\bu$ such that 
 \begin{enumerate}
  \item $\bu$ is progressively measurable,
   \item $\bu$ belongs to $C([0,T];\h)\cap L^2(0,T,\ve)$ almost surely,
   \item for all $t\in [0,T]$, almost surely 
   \begin{equation*}
    (\bu(t), \phi)+\nu \int\langle \hsa \bu(s), \hsa \phi\rangle ds+\int_0^t \langle B(\bu(s),\bu(s)), \phi\rangle ds=\bu_0 +\int_0^t\langle \phi,Q(\bu(s))dW(s)\rangle, 
   \end{equation*}
for any $\phi\in \ve$.
 \end{enumerate}
 
\end{Def}
We state the following theorem which was proved in \cite{IC+AM-1} (see also \cite{Bensoussan,Flandoli_Gatarek_95}).
 \begin{thm}
  Let $\bu_0$ be a $\h$-valued $\mathcal{F}_0$-measurable such that $\mathbb{E}\lvert \bu_0\rvert^4<\infty$. Assume that \eqref{HYP-Q} holds. 
  Then \eqref{NSE} has a unique solution $\bu$ in the sense of the above definition. Moreover,  for any $p\in \{2,4\}$ and $T>0$
  there exists $C>0$ such that
  \begin{equation}\label{EST-NSE}
  \mathbf{E}\left(\sup_{t\in [0,T]}\lvert \bu(t)\rvert^p +\int_0^T \lvert \bu(s)\rvert^{p-2}\lvert \hsa \bu(s)\rvert^2 ds \right)<C(1+\lvert \bu_0\rvert^4) 
  \end{equation}
 \end{thm}
\subsection{A priori estimates for the Leray-$\alpha$ model}\label{AP-EST-LAM}

 The Leray-$\alpha$ model was introduced and analyzed in \cite{CHOT}. Since, then it has been extensively studied; we refer to \cite{Titi} and references therein for a brief historical
description and review of results.
% from the mathematical and computational studies of this system.
It is worth noticing that the Leray-$\alpha$ model is a particular example of a more general regularization used by Leray in his seminal work, \cite{Leray},  in the context of establishing
the existence of solutions for the 2D and 3D NSE.

The stochastic Leray-$\alpha$ model was studied in \cite{IC+AM-1, IC+AM-2, DEUGOUE2}. For the inviscid case, we refer to the recent work \cite{BBF} where the uniqueness of solutions were investigated. 
The following definition of solutions to \eqref{LAM} is taken from \cite{DEUGOUE2} (see also \cite{IC+AM-1}).
\begin{Def}
 A weak solution to \eqref{LAM} is a stochastic process $\bu^\alpha$ such that 
 \begin{enumerate}
  \item $\bu^\alpha$ is progressively measurable,
   \item $\bv^\alpha$, with $\bv^\alpha:=(I+\alpha^2 \sA) \bua$, belongs to $C([0,T];\h)\cap L^2(0,T,\ve)$ almost surely,
   \item for all $t\in [0,T]$, almost surely 
   \begin{equation*}
    (\bv^\alpha(t), \phi)+\nu \int\langle \hsa \bv^\alpha(s), \hsa \phi\rangle ds+\int_0^t \langle B(\bua(s),\bv^\alpha(s)), \phi\rangle ds=\bv^\alpha_0 +\int_0^t\langle \phi,Q(\bu^\alpha(s))dW(s)\rangle, 
   \end{equation*}
for any $\phi\in \ve$.
 \end{enumerate}
 
\end{Def}
We state the following theorem which was proved in \cite{IC+AM-1} (see also \cite{DEUGOUE2}).
 \begin{thm}
  Let $\bu_0$ be a $D(\sA)$-valued $\mathcal{F}_0$-measurable such that $\mathbb{E}\lvert \sA \bu_0\rvert^4<\infty$. Assume that \eqref{HYP-Q} holds. 
  Then for any $\alpha>0$ the system \eqref{LAM} has a unique solution $\bua$ in the sense of the above definition. Moreover,  for any $p\in \{2,4\}$ and $T>0$
  there exists $C>0$ such that
  \begin{equation}\label{EST-LAM-A0}
  \mathbf{E}\left(\sup_{t\in [0,T]}\lvert \bv^\alpha(t)\rvert^p +\int_0^T \lvert \bv^\alpha(s)\rvert^{p-2}\lvert \hsa \bv^\alpha(s)\rvert^2 ds \right)<C(1+\lvert(I+\alpha^2 \sA) \bu_0\rvert^4) 
  \end{equation}
 \end{thm}

 As stated in \cite{DEUGOUE2},  the constant $C$ above depends on $\alpha$ and may explode as $\alpha$ tends to zero.  The uniform estimates (wrt $\alpha$) obtained in
 \cite{DEUGOUE2} are not helpful  for our analysis. Our aim in this subsection is to derive several a priori estimates for the stochastic Leray-$\alpha$ model \eqref{LAM}.
These estimates summarized in the following two proposition, will be used
in Section \ref{RATE} to derive a rate of convergence of the stochastic Leray-$\alpha$ model to the stochastic Navier-Stokes
equations.

We start with some estimates in the weak norms of the solution $\bua$, these are refinements of estimates obtained in \cite{DEUGOUE2}.
\begin{prop}\label{PROp-EST-LAM-0}
 Let $\bu_0$ be a $\mathcal{F}_0$-measurable random variable such that $\me \lvert \bu_0+\sA\bu_0\rvert^4<\infty.$ Assume that
 the set of hypotheses stated in Assumption \ref{HYP-Q} holds. Then, there exists a constant $C>0$ such that for any $\alpha\in (0,1)$
we have
 \begin{align}
   \me \sup_{t\in [0,T] } \biggl[\lvert \bu^\alpha(t)\rvert^2+2\alpha^2 \lvert \hsa \bu^\alpha(t) \rvert^2
   +\alpha^4 \lvert \sA\bua(t)\rvert^2\biggr]^2 \le \mathfrak{K}_0,\label{WAP-EST-LAM-1}\\
    \me \int_0^T \lvert \hsa \bu^\alpha (s)\rvert^2 \lvert \bu^\alpha(s)\rvert^2 ds\le \mathfrak{K}_0, \label{WAP-EST-LAM-2}\\
   2 \alpha^2 \me \int_0^T\lvert \bua(s)\rvert^2\Big( \lvert \sA\bu^\alpha (s)\rvert^2+\frac{\alpha^2}{2} \lvert \sA^\frac32 \bua(s)\rvert^2 \Big) ds\le \mathfrak{K}_0 \label{WAP-EST-LAM-3},\\
   2 \alpha^2 \me \int_0^T \lvert \hsa \bu^\alpha(s)\rvert^2\Big( \lvert \hsa \bu^\alpha(s)\rvert^2+2\alpha^2 \lvert \sA \bu^\alpha(s)\rvert^2+ \alpha^4 \lvert \sA^\frac32 \bu^\alpha(s)\rvert^2 \Big)ds\le \mathfrak{K}_0, \label{WAP-EST-LAM-4}\\
  \alpha^4 \me \int_0^T \lvert \sA\bu^\alpha(s) \rvert^2\Big( \lvert \hsa \bu^\alpha(s)\rvert^2+2\alpha^2 \lvert \sA \bu^\alpha(s)\rvert^2+ \alpha^4 \lvert \sA^\frac32 \bu^\alpha(s)\rvert^2 \Big)
  ds \le \mathfrak{K}_0,
\label{WAP-EST-LAM-5}
  \end{align}
  where $$\mathfrak{K}_0:=\left(\me \lvert \bu_0+\sA \bu_0\rvert^4+CT\right) \left(1+Ce^{CT}\right).$$
\end{prop}
\begin{proof}
  For any positive integer $n\in \mathbb{N}$, we will consider the Galerkin approximation of \eqref{LAM} which is a system of
 SDEs in $\h_n$
 \begin{subequations}\label{GAL-AP-LAM}
  \begin{align}
    d\bv_n^\alpha(t) +[\nu\sA\bv_n^\alpha(t) +B(\bu_n^\alpha(t), \bv_n^\alpha(t))]dt= P_n Q{(\buna(t))}dW(t), \, t\in (0,T]
    \label{GLAM-1}\\
 \bu_n^\alpha+\alpha^2 \sA\bu_n^\alpha=\bv_n^\alpha,\label{GLAM-2}\\
 \bu_n^\alpha(0)=\bu_{0n},
  \end{align}
 \end{subequations}
where $\bu_{0n}=P_n \bu_0$. Let $\Psi(\cdot)$ be a mapping defined
on $\h_n$ defined by $\Psi(\cdot):=\lvert \cdot \rvert^4$. The
mapping $\Psi(\cdot)$ is twice Fr\'echet differentiable with first
and second derivative defined by
\begin{align*}
 \Psi^\prime(\bu)[\mathbf{f}]= & 4 \lvert \bu\rvert^2 \langle \bu, \mathbf{f}\rangle,\\
 \Psi^{\prime \prime}(\bu)[\mathbf{f},\mathbf{g}]=  & 4 \lvert \bu\rvert^2 \langle \mathbf{g}, \mathbf{f}\rangle+8 \langle  \bu, \mathbf{g}\rangle \langle \bu, \mathbf{f}\rangle,
\end{align*}
for any $\bu, \mathbf{f}, \mathbf{g} \in \h_n$. In particular, the last identity implies that
\begin{equation*}
 \Psi^{\prime \prime}(\bu)[\mathbf{f},\mathbf{f}]\le 12 \lvert \bu\rvert^2 \lvert \mathbf{f}\rvert^2,
\end{equation*}
for any $\bu, \mathbf{f}\in \h_n$. Therefore by It\^o's formula to
$\Psi(\bva_n):=\lvert \bva_n(t)\rvert^4 $ we obtain
\begin{equation*}
 \begin{split}
  d \lvert \bva_n(t)\rvert^4 +4 \lvert \bva_n(t)\rvert^2 \biggl[\nu  \langle \sA \bva_n(t)+B(\bua_n(t), \bva_n(t)), \bva_n(t)\rangle \biggr]dt\\
  \le C \lvert \bva_n(t) \rvert^2
  \Lve Q{(\bua_n(t))}\Rve^2_{\mathcal{L}_2(\rK,\h)} dt + 4\lvert \bva_n(t)\rvert^2 \langle \bva_n(t), P_n Q{(\bua_n(t))} dW(t)\rangle.
 \end{split}
\end{equation*}
By using the identity \eqref{BIL-B4}, the Cauchy's inequality and Assumption \ref{HYP-Q}-\eqref{HYP-Q-i} along with Remark \ref{Rem-HYP}-\eqref{Rem-HYP-i} we infer the existence of a constant $c>0$ such that
\begin{equation}\label{WAP-EST}
 \begin{split}
  d \lvert \bva_n(t)\rvert^4 +4\nu \lvert \bva_n(t)\rvert^2 \langle \sA \bva_n(t), \bva_n(t)\rangle dt\le c \lvert \bva_n(t) \rvert^4 dt + c\, dt
  \\ + 4\lvert \bva_n(t)\rvert^2 \langle \bva_n(t), P_n Q{(\bua_n(t))} dW(t)\rangle.
 \end{split}
\end{equation}
Since, by definition of $\bva$,  $\lvert \bua\rvert \le c \lvert \bva \rvert$ we deduce from Assumption \ref{HYP-Q}-\eqref{HYP-Q-i} that
$$  \Lve Q{(\bua_n(s))}\Rve^2_{\mathcal{L}_2(\rK,\h)} \le c (1+\lvert \bva\rvert)^2 .$$
Now, using Burkholder-Davis-Gundy and Cauchy-Schwarz inequalities, we deduce that
\begin{align*}
 \me \sup_{s \in [0,t]}\biggl\vert \int_0^s \lvert \bva_n(s)\rvert^2 \langle \bva_n(s), P_n Q{(\bua_n(s))} dW(s)\rangle\biggr\vert \le & c  \me \biggl(\int_0^t
 \lvert \bva_n(s)\rvert^4 \lvert \bva_n(s)\rvert^2 \Lve Q{(\bua_n(s))}\Rve^2_{\mathcal{L}_2(\rK,\h)} ds\biggr)^\frac12,\\
 \le & c \me \biggl[\sup_{s\in [0,t]}\lvert \bva_n(s)\rvert^2 \biggl(\int_0^t c(1+\lvert \bva_n(s)\rvert)^4 ds\biggr)^\frac12\biggr]\\
 \le & \frac12 \me \sup_{s\in [0,t]} \lvert \bva_n(s)\rvert^4+c T+ c \me \int_0^t \lvert \bva_n(s)\rvert^4 ds.
\end{align*}
From this last estimate and \eqref{WAP-EST} we derive that there exists $C>0$ such that
\begin{equation*}
 \begin{split}
\me \sup_{s\in [0,t]}\lvert \bva_n(s)\rvert^4 +8\nu \me \int_0^t \lvert \bva_n(s) \rvert^2 \langle \sA\bva_n(s), \bva_n(s)\rangle ds\le  \me \lvert \bva(0)\rvert^4+CT +C
\me \int_0^t \lvert \bva_n(s)\rvert^2 ds.
 \end{split}
\end{equation*}
Since $\langle \sA\bva_n(s),\bva_n(s)\rangle =\lvert \hsa \bva_n(s)\rvert^2$ is nonnegative, and using  the Gronwall's inequality, we deduce that
\begin{equation*}
 \me \sup_{s\in [0,t]}\lvert \bva_n(s)\rvert^4 +8\nu \me \int_0^t \lvert \bva_n(s) \rvert^2 \langle \sA\bva_n(s), \bva_n(s)\rangle ds\le (\me \lvert \bva(0)\rvert^4+CT )(1+Ce^{CT})
\end{equation*}
Since $\alpha$ tend to 0, we can assume that $\alpha\in (0,1)$. Therefore, by lower semicontinuity of the norm and the fact that $\lvert \bva(0)\rvert^4\le \lvert \bu_0+\sA\bu_0\rvert^4$, we infer that as $n\rightarrow \infty$
\begin{equation}\label{WAP-GAL-EST-Fin}
 \me \sup_{s\in [0,t]}\lvert \bva(s)\rvert^4 +8\nu \me \int_0^t \lvert \bva(s) \rvert^2 \langle \sA\bva(s), \bva(s)\rangle ds\le \mathfrak{K}_0
\end{equation}
where $\mathfrak{K}_0:=(\me \lvert \bu_0+\sA\bu_0\rvert^4+CT )(1+Ce^{CT})$. Since
\begin{align*}
 \lvert \bva\rvert^2=\lvert \bua\rvert^2+2\alpha^2 \lvert \hsa \bua\rvert^2+\alpha^4 \lvert \sA\bua\rvert^2,\\
 \langle \sA\bva, \bva\rangle= \lvert \hsa \bua\rvert^2+2\alpha^2 \lvert \sA\bua\rvert^2+\alpha^4 \lvert \sA^\frac32 \bua\rvert^2,
\end{align*}
we deduce from \eqref{WAP-GAL-EST-Fin} that the five estimates \eqref{WAP-EST-LAM-1}-\eqref{WAP-EST-LAM-5} hold.
\end{proof}
As a consequence of the estimate \eqref{WAP-EST-LAM-2} and the Gagliardo-Nirenberg's inequality
$$\Lve \bv \Rve_{\el^4}\le c \lvert \bv \rvert^\frac12 \lvert \hsa \bv \rvert^\frac12,$$ we state the following corollary.
\begin{cor}\label{COR-WAP}
 Under the assumptions of Proposition \ref{PROp-EST-LAM-0} there exists $C>0$ such that for any $\alpha \in (0,1)$ we have
 \begin{equation}\label{WAP-EST-LAM-6}
  \me \int_0^T \Lve \bua(s)\Rve^4 ds \le (\me \lvert \bu_0+\sA\bu_0\rvert^4+CT )(1+Ce^{CT}).
 \end{equation}
\end{cor}
Now we state several important estimates for the norm of $\bua$ in stronger norms.
\begin{prop}\label{PROP-EST-LAM}
 Let Assumption \ref{HYP-Q} holds and let $\bu_0$ be a $\mathcal{F}_0$-measurable random variable such that
 $\me(\vert \hsa \bu_0\rvert^2+\vert \sA \bu_0\vert^2)^2<\infty$. Then, there exists a generic constant $K_0>0$ such that for any $\alpha \in (0,1)$ we have
 %\begin{equation}
  \begin{align}
   \me \sup_{t\in [0,T] } \biggl[\lvert \hsa \bu^\alpha(t)\rvert^4+\alpha^4 \lvert \sA \bu^\alpha(t) \rvert^4
   +2\alpha^2 \lvert \hsa \bu^\alpha(t) \rvert^2\lvert \sA\bu(t)\rvert^2\biggr] \le K_0,\label{AP-EST-LAM-1}\\
    \me \int_0^T \lvert \hsa \bu^\alpha (s)\rvert^2 \lvert \sA\bu^\alpha(s)\rvert^2 ds\le K_0, \label{AP-EST-LAM-2}\\
   \alpha^2 \me \int_0^T \lvert \sA\bu^\alpha (s)\rvert^4 ds\le K_0 \label{AP-EST-LAM-3},\\
    \alpha^2 \me \int_0^T \lvert \hsa \bu^\alpha(s)\rvert^2 \lvert \sA^\frac32 \bu^\alpha(s)\rvert^2 ds\le K_0, \label{AP-EST-LAM-4}\\
  \alpha^4 \me \int_0^T \lvert \sA\bu^\alpha(s) \rvert^2 \lvert \sA^\frac32 \bu^\alpha(s)\rvert^2 ds \le K_0,
\label{AP-EST-LAM-5}
  \end{align}
 %\end{equation}
 where $$K_0:= [2\me(\vert \hsa \bu_0\rvert^2+\vert \sA \bu_0\vert^2)^2+CT]\cdot [1+C e^{CT}].  $$
\end{prop}
\begin{proof}
As in the proof of Proposition \ref{PROp-EST-LAM-0} we still consider the soltuion $\bva_n$ (or $\bua_n$) of the $n$-th Galerkin approximation of \eqref{LAM} defined by the system of SDEs
\eqref{GAL-AP-LAM}. Let $\na$ be the self-adjoint and positve definite operator defined by $\na \bv =\ipsa \bv $ for any $\bv \in \h$. It is well-known that $\na^{-1}$ with domain $D(\sA)$
is also positive definite and self-adjoint on $\h$. Thus, the fractional powers $\na^\frac12 $ and $\na^{-\frac12}$ are also self-adjoint.
Since $\bva=\na^{-1}\bua$ it follows from \eqref{GLAM-1} that
\begin{equation}\label{GLAM-3}
d\bu_n^\alpha(t) +[\nu\sA\bu_n^\alpha(t) +\na B(\bu_n^\alpha(t), \bv_n^\alpha(t))]dt= \na P_n Q{(\bua_n(t))}dW(t),
\end{equation}

Let $\Phi:D(\sA)\rightarrow [0,\infty)$ be the mapping defined by $\Phi(\bv)=\langle \sA\bv, \na^{-1}\bv\rangle$ for any $\bv \in D(\sA)$. It is not difficut to show that $\Phi(\cdot)$ is
twice Fr\'echet differentiable and its first and second derivatives satisfy
\begin{align*}
 \Phi^\prime(\bua)[\mathbf{f}]=& \langle \sA\bua, \na^{-1}\mathbf{f}\rangle+\langle \na^{-1}\bua,\sA \mathbf{f}\rangle\\
 =& \langle A\bva, \mathbf{f}\rangle+ \langle \bva, \sA \mathbf{f}\rangle,\\
 \Phi^{\prime \prime}(\bua)[\mathbf{f},\mathbf{g}]=&\langle \sA \mathbf{g}, \na^{-1}\mathbf{f}\rangle+\langle \sA \mathbf{f}, \na^{-1} \mathbf{g}\rangle,
\end{align*}
for any $\mathbf{f}, \mathbf{g} \in D(A)$. In particular, the last identity and $\hsa$ and $\na^{-\frac12}$
being self-adjoint imply  that
\begin{equation*}
 \Phi^{\prime \prime}(\bua)[\mathbf{f},\mathbf{f}]=2\lvert \hsa \na^{-\frac12}\mathbf{f}\rvert^2.
\end{equation*}
Therefore, using the It\^o formula for $\Phi(\bua_n)$ and \eqref{GLAM-3}, we derive that there exists $c>0$
\begin{equation*}
\begin{split}
 d\Phi(\bua_n(t))\le \Phi^\prime(\bua_n(t))[-\nu\sA\bua_n(t)-\na B(\bua_n(t),\bva_n(t))]dt\\+ \Phi^\prime(\bua_n(t))[\na P_n Q{(\bua_n(t))}]dW(t)\\+
 c \Vert \hsa \na^{-\frac12}\na P_n Q{(\bua_n(t))}\Vert^2_{\mathcal{L}_2(\rK,\h)} dt.
 \end{split}
\end{equation*}
Referring to the equation for $\Phi^\prime(\bua)[\cdot]$ we see that
\begin{equation*}
\begin{split}
\Phi^\prime(\bua_n(t))[-\nu\sA\bua_n(t)-\na B(\bua_n(t),\bva_n(t))]=\langle \sA \bva_n(t), -\nu\sA\bua_n(t)-\na B(\bua_n(t),\bva_n(t))\\
+\langle \bva_n(t), \sA[-\nu\sA\bua_n(t)-\na B(\bua_n(t),\bva_n(t))]\rangle,
\end{split}
\end{equation*}
and
\begin{equation*}
 \begin{split}
 \Phi^\prime(\bua_n(t))[\na P_n Q{(\bua_n(t))}]dW(t)=\langle \sA\bva_n(t), \na P_n Q{(\bua_n(t))} dW(t)\rangle\\+\langle \bva_n(t), \sA [\na P_n Q{(\bua_n(t))}dW(t)]\rangle.
 \end{split}
\end{equation*}
Hence
\begin{equation}\label{GLAM-4}
 d\langle \sA\buna(t), \bvan(t)\rangle\le \langle \sA\buna(t), d\bvan(t)\rangle+\langle \bvan(t),
 d\sA\buna(t)\rangle+
 c \Vert \hsa \na^{-\frac12}\na P_n Q{(\bua_n(t))}\Vert^2_{\mathcal{L}_2(\rK,\h)} dt.
\end{equation}
%where $[\bvan, \sA\buna]_t$ is the quadratic covariation of the $\h_n$-valued stochastic processes $\bvan$ and $\sA\buna$.
First, we estimate the term $\langle \sA\buna(t), d\bvan(t)\rangle$. We derive from \eqref{GLAM-1}
that
\begin{equation}\label{GLAM-5}
 \begin{split}
  \langle \sA\buna(t), d\bvan(t)\rangle  =&
  [-\langle \sA\buna(t), \sA\bvan(t) \rangle-\langle B(\buna(t), \bvan(t)), \sA\buna(t)\rangle ]dt \\
  & \quad + \langle
  \sA\buna(t), P_n QdW(t)\rangle.
 \end{split}
\end{equation}
Recalling the definition of $\bvan$ we derive that
$$\langle \sA\buna(t), \sA\bvan(t) \rangle=\lvert \sA\buna(t)\rvert^2+\alpha^2 \lvert \sA^\frac32 \buna(t)\rvert^2. $$
Owing to the definition of $\bvan$, \eqref{BIL-B4} and \eqref{BIL-B5} we have
\begin{equation}\label{GLAM-6}
 \begin{split}
  \langle B(\buna(t), \bvan(t)), \sA\buna(t)\rangle=
\alpha^2 \langle B(\buna(t), \sA\buna(t), \sA\buna(t)\rangle
 + \langle B(\buna(t), \buna(t), \sA\buna(t) \rangle\\
 =0.
 \end{split}
\end{equation}
Therefore we derive from \eqref{GLAM-5}-\eqref{GLAM-6} that
\begin{equation}\label{GLAM-7}
 \langle \sA\buna(t), d\bvan(t)\rangle=-\lvert \sA\buna(t)\rvert^2 - \alpha^2 \lvert \sA^\frac32 \buna(t)\rvert^2 +
  \langle \sA\buna(t), P_n QdW(t)\rangle.
\end{equation}

Second, we treat the term $\langle \bvan(t), d\sA\buna(t)\rangle$, but before proceeding further we observe that
\begin{equation*}
 \sA\na=\frac1{\alpha^2}[\Id -\na]
\end{equation*}
from which it follows that
\begin{equation*}
\begin{split}
\langle  \sA\na B(\buna(t),\bvan(t)), \bvan(t)\rangle = & \frac1{\alpha^2}\langle B(\buna(t),\bvan(t), \bvan(t) \rangle
-\frac1{\alpha^2}\langle \na B(\buna(t),\bvan(t)), \bvan(t)\rangle\\
=& -\frac1{\alpha^2}\langle \na B(\buna(t),\bvan(t)), \bvan(t)\rangle,
\end{split}
\end{equation*}
where \eqref{BIL-B3} was used to derive the last line. Since $\bvan=\na^{-1} \buna$, we obtain that
\begin{equation*}
\begin{split}
 \langle \na B(\buna(t),\bvan(t)), \bvan(t)\rangle= & \langle B(\buna(t),\bvan(t)),\buna(t)\rangle\\
 = & \langle B(\buna(t),\buna(t)), \buna(t)\rangle+\alpha^2 \langle B(\buna(t), \sA\buna(t)), \buna(t)\rangle,
\end{split}
\end{equation*}
Owing to this last identity, \eqref{BIL-B3}-\eqref{BIL-B5} we infer that
\begin{equation}\label{GLAM-8}
\langle \sA \na B(\buna(t),\bvan(t)), \bvan(t)\rangle=0.
\end{equation}
Since
\begin{equation*}
 d\sA\buna(t)=[-\sA^2 \buna(t)-\sA\na B(\buna(t),\bvan(t))]dt+\sA\na P_n Q{(\bua_n(t))}dW(t),
\end{equation*}
 it follows by invoking \eqref{GLAM-8} and using the definition of $\bvan$  that
 \begin{equation*}
 \begin{split}
  \langle \bvan(t),d\sA\buna(t)\rangle=-\langle \buna(t),\sA^2 \buna(t)\rangle
  -\alpha^2 \langle \sA\buna(t), \sA^2 \buna(t)\rangle \\
  +\langle \bvan(t),\sA\na P_n Q{(\bua_n(t))} dW(t)\rangle.
  \end{split}
 \end{equation*}
From this latter identity we easily derive that
\begin{equation}\label{GLAM-9}
 \langle \bvan(t),d\sA\buna(t)\rangle=-\lvert \sA\buna(t)\rvert^2 -\alpha^2 \lvert \sA^\frac32 \buna(t)\rvert^2 + \langle
 \bvan(t), \sA\na P_n Q{(\bua_n(t))}  dW(t)\rangle.
\end{equation}
%  Now we derive an explicit formula for the quadratic covariation $[\bvan, \sA\buna]_t$.
%  We have
%  \begin{align}
%   [\bvan,\sA\buna]_t= & \biggl[\int \sum_{i=1}^\infty P_n Q{(\bua_n(t))}  e_i d\beta_i, \sum_{j=1}^\infty \sA\na P_n Q{(\bua_n(t))}  e_j d\beta_j\biggr]\nonumber\\
%   =& \sum_{i=1}\infty \int_0^t\langle P_n Q{(\bua_n(t))}  e_i, \sA \na P_n Q{(\bua_n(t))}  e_i\rangle ds \nonumber \\
%   =& \sum_{i=1}^\infty \int_0^t\langle \hsa P_n Q{(\bua_n(t))}  e_i, \hsa \na P_n Q{(\bua_n(t))}  e_i\rangle ds.\nonumber
%  \end{align}
% Since $\na=\hipsa \hipsa$ and $\na^\frac12=\hipsa$ we infer that
% \begin{equation}\label{GLAM-10}Deugoue, Gabriel; Sango, Mamadou Weak solutions to stochastic 3D Navier-Stokes-α model of turbulence: α-asymptotic behavior. J. Math. Anal. Appl. 384 (2011), no. 1, 49–62.
%   [\bvan,\sA\buna]_t= \int_0^t \Lve \hsa \na^\frac12 P_n Q{(\bua_n(t))}  \Rve^2_{\mathcal{L}_2(H)} ds.
% \end{equation}
Plugging \eqref{GLAM-7} and \eqref{GLAM-9} in \eqref{GLAM-4} implies that
\begin{equation}\label{GLAM-11}
 \begin{split}
  \lvert\hsa \buna(t)\rvert^2+\alpha^2 \lvert \sA \buna(t)\rvert^2 - \lvert \hsa \bu_{0n}\rvert^2 -
  \alpha ^2 \lvert \sA \bu_{0n} \rvert^2 +2\int_0^t \Big(\lvert \sA\buna(s)\rvert^2+
  \alpha^2 \lvert \sA^\frac32 \buna(t)\rvert^2 \Big)ds  \\ \le c
  \int_0^t \Lve \hsa \na^\frac12 P_n Q{(\bua_n(t))}  \Rve^2_{\mathcal{L}_2(H)} ds+
  2 \int_0^t\langle \hsa \buna(s), \hsa P_n Q{(\bua_n(t))}  dW(s)\rangle,
 \end{split}
\end{equation}
where we have used the fact that
%Since $ \bvan=\na^{-1} \buna$ and $\na$ is self-adjoint we obtain that
\begin{equation*}
\begin{split}
 \langle \bvan(s), \sA\na P_n Q{(\bua_n(t))}  dW(s)\rangle= \langle A \buna(s), P_n Q{(\bua_n(t))}  dW(s)\rangle\\
 =\langle \hsa \buna(s), \hsa P_n Q{(\bua_n(t))}  dW(s)\rangle.
 \end{split}
\end{equation*}
% Therefore,
% \begin{equation*}
%  \begin{split}
%   \lvert\hsa \buna(t)\rvert^2+\alpha^2 \lvert \sA \buna(t)\rvert^2 - \lvert \hsa \bu_{0n}\rvert^2 -
%   \alpha ^2 \lvert \sA \bu_{0n} \rvert^2 +2\int_0^t \Big(\lvert \sA\buna(s)\rvert^2+
%   \alpha^2 \lvert \sA^\frac32 \buna(t)\rvert^2 \Big)ds  \\ = \int_0^t \Lve \hsa \na^\frac12 P_n Q{(\bua_n(t))}  \Rve^2_{\mathcal{L}_2(H)} ds+
%   2 \int_0^t\langle \sA\buna(s), P_n Q{(\bua_n(t))}  dW(s)\rangle.
%  \end{split}
% \end{equation*}
%%%%%%%%%%%%%%%%%%%%%%%%%%%%%%
By the Burkholder-Davis-Gundy, Cauchy-Schwarz, Cauchy inequalities and Assumption \ref{HYP-Q}-\eqref{HYP-Q-ii} along with
Remark \ref{Rem-HYP}-\eqref{Rem-HYP-ii} we derive that
\begin{align}
%\begin{split}
 \me \sup_{r\in [0,t]}\biggl\vert \int_0^r \langle \sA\buna(s), P_n Q{(\bua_n(t))}  dW(s)\rangle\biggr\vert\nonumber \\
 \le c
 \me \biggl[\sup_{s\in [0,t]}\lvert \hsa \buna(s)\rvert \times \biggl(\int_0^t \Lve \hsa Q{(\buna(s))}\Rve_{\mathcal{L}_2(\rK,\h)} ds\biggr)^\frac12\biggr] \nonumber \\
 \le \frac14 \me \Big(\sup_{s\in [0,t]}[\lvert \hsa\buna(s)\rvert^2+\alpha^2 \lvert \sA \buna(s)\rvert^2]\Big)+
  c^2 \me \int_0^t \Lve \hsa Q{(\buna(s))}\Rve_{\mathcal{L}_2(\rK,\h)}^2 ds\nonumber \\
  \le \frac14 \me \Big(\sup_{s\in [0,t]}[\lvert \hsa\buna(s)\rvert^2+\alpha^2 \lvert \sA \buna(s)\rvert^2]\Big)+ c^2 \ell^2_3 \me\int_0^t (1+\lvert \hsa \buna(s)\rvert)^2 ds\nonumber \\
  \le \frac14 \me \Big(\sup_{s\in [0,t]}[\lvert \hsa\buna(s)\rvert^2+\alpha^2 \lvert \sA \buna(s)\rvert^2]\Big)+ c T+c \me\int_0^t [\lvert \hsa \buna(s)\rvert^2+\alpha^2 \lvert \sA \buna(s)\rvert^2] ds.
  \label{1ST-BDG}
 %\end{split}
\end{align}
Since $\na^\frac12 $ is self-adjoint and $\Lve \na \Rve_{\mathcal{L}(H)}\le 1$ we infer that
\begin{equation}\label{NASQ}
 \Lve \na^\frac12 \Rve_{\mathcal{L}(\h)}\le 1.
\end{equation}
Thus, it follows from Assumption \ref{HYP-Q}-\eqref{HYP-Q-ii} along with
Remark \ref{Rem-HYP}-\eqref{Rem-HYP-ii}
\begin{equation}\label{GLAM-12}
 \int_0^t \Lve \hsa \na^\frac12 P_n Q{(\bua_n(t))}  \Rve^2_{\mathcal{L}_2(\rK,\h)} ds\le c T+c \me\int_0^t [\lvert \hsa \buna(s)\rvert^2+\alpha^2 \lvert \sA \buna(s)\rvert^2] ds.
\end{equation}
Hence, the calculations between \eqref{GLAM-11} and \eqref{GLAM-12} yield
\begin{equation*}
 \begin{split}
  \me \big(\sup_{t\in[0,T]}[ \lvert\hsa \buna(t)\rvert^2+\alpha^2 \lvert \sA \buna(t)\rvert^2]\big)
  +4\me \int_0^t \Big(\lvert \sA\buna(s)\rvert^2+
  \alpha^2 \lvert \sA^\frac32 \buna(t)\rvert^2 \Big)ds \\
  \le   c T+c \me\int_0^t [\lvert \hsa \buna(s)\rvert^2+\alpha^2 \lvert \sA \buna(s)\rvert^2] ds
  + 2\lvert \hsa \bu_{0n}\rvert^2 +
  2 \alpha ^2 \lvert \sA \bu_{0n} \rvert^2.
 \end{split}
\end{equation*}
Since $\alpha \in (0,1)$ we derive from the
last estimate that for any
 $\alpha \in (0,1)$ and $n \in \mathbb{N}$
  \begin{equation*}
 \begin{split}
  \me \big(\sup_{t\in[0,T]}[ \lvert\hsa \buna(t)\rvert^2+\alpha^2 \lvert \sA \buna(t)\rvert^2]\big)
  +2\me \int_0^t \Big(\lvert \sA\buna(s)\rvert^2+
  \alpha^2 \lvert \sA^\frac32 \buna(t)\rvert^2 \Big)ds\\
  \le c T+c \me\int_0^t [\lvert \hsa \buna(s)\rvert^2+\alpha^2 \lvert \sA \buna(s)\rvert^2] ds+ 2\lvert \hsa \bu_{0}\rvert^2 +2 \lvert \sA \bu_{0} \rvert^2.
 \end{split}
\end{equation*}
Now it follows from the Gronwall lemma that there exists $C>0$ such that for any $\alpha \in (0,1)$ and $n \in \mathbb{N}$ we have
\begin{equation*}
 \begin{split}
  \me \big(\sup_{t\in[0,T]}[ \lvert\hsa \buna(t)\rvert^2+\alpha^2 \lvert \sA \buna(t)\rvert^2]\big)
  +2\me \int_0^t \Big(\lvert \sA\buna(s)\rvert^2+
  \alpha^2 \lvert \sA^\frac32 \buna(t)\rvert^2 \Big)ds\\
  \le [CT+ 2\lvert \hsa \bu_{0}\rvert^2 +2 \lvert \sA \bu_{0} \rvert^2](1+C e^{CT}).
 \end{split}
\end{equation*}
As $n\rightarrow \infty$, by lower semicontinuity we
deduce that
 \begin{equation}\label{GLAM-13}
 \begin{split}
  \me \big(\sup_{t\in[0,T]}[ \lvert\hsa \bua(t)\rvert^2+\alpha^2 \lvert \sA \bua(t)\rvert^2]\big)
  +2\me \int_0^t \Big(\lvert \sA\bua(s)\rvert^2+
  \alpha^2 \lvert \sA^\frac32 \bua(t)\rvert^2 \Big)ds \le K.
 \end{split}
\end{equation}
where $$K:=[CT+ 2\lvert \hsa \bu_{0}\rvert^2 +2 \lvert \sA \bu_{0} \rvert^2](1+C e^{CT}).$$

Now, let $y^\alpha_n (t)=\lvert \hsa \buna(t)\rvert^2 +\alpha^2 \lvert \sA\buna(t)\rvert^2$. Observing that
$$\langle \bvan(t), \sA\buna(t)\rangle=\lvert \hsa \buna(t)\rvert^2 +\alpha^2 \lvert \sA\buna(t)\rvert^2,$$
 we see that $[y^\alpha_n(t)]^2=[\Phi(\buna(t))]^2$. Therefore,
 from It\^o's formula and \eqref{NASQ} we deduce that
 \begin{equation*}
  \begin{split}
   d([\by]^2(t))\le [-4\by(t) (\lvert \sA\buna(t)\rvert^2+\alpha^2 \lvert \sA^\frac32 \buna(t)\rvert^2) + 2c \by(t) \Lve
   \hsa Q{(\buna(t))} \Rve^2_{\mathcal{L}_2(\rK,\h)}\\ + c \lvert \hsa \buna(t)\rvert^2 \Lve
   \hsa Q{(\buna(t))} \Rve^2_{\mathcal{L}_2(\rK,\h)} ]dt +4\by(t) \langle \hsa \buna(t), \hsa Q{(\buna(t))} dW(t)\rangle.
  \end{split}
 \end{equation*}
From this last inequality we infer that
\begin{equation}\label{FORMER}
 \begin{split}
  \me \sup_{s\in [0,t]} [\by]^2(s)- [\by(0)]^2+4\int_0^t \by(s)
  [\lvert \sA\buna(t)\rvert^2+\alpha^2 \lvert \sA^\frac32 \buna(t)\rvert^2]ds\\
  \le  4 \me \sup_{r\in [0,t]} \biggl \lvert \int_0^r \by(s) \langle \hsa \buna(t), \hsa Q{(\buna(s))} dW(s)\rangle\biggl\vert \\
   +c \me \int_0^t \by(s) \Lve \hsa Q{(\buna(s))} \Rve^2_{\mathcal{L}_2(\rK,\h)}ds.
 \end{split}
\end{equation}
Thanks to Remark \ref{Rem-HYP}-\eqref{Rem-HYP-ii} we easily derive that 
\begin{equation}\label{NOT-1}
 \by(s) \Lve \hsa Q{(\buna(s))} \Rve^2_{\mathcal{L}_2(\rK,\h)}\le C(1 + [\by(s)]^2)   
\end{equation}
Now, arguing as in the proof of \eqref{1ST-BDG} and using this last inequality we obtain the following estimates
\begin{align*}
 4 \me \sup_{r\in [0,t]} \biggl\lvert \int_0^r \by(s) \langle \hsa \buna(t), \hsa Q{(\buna(s))} dW(s)\rangle \biggl\lvert  \\
 \le \frac12 \me \sup_{s\in [0,t]} [\by(s) \lvert \hsa \buna(s)\rvert^2] 
 +  c \me \int_0^t \by(s)\Lve \hsa Q{(\buna(s))} \Rve^2_{\mathcal{L}_4(\h)} ds\\
 \le \frac12 \me \sup_{s\in [0,t]} [\by(s)\times \left(\lvert \hsa \buna(s)\rvert^2(s)+\alpha^2\vert \sA \buna(s)\rvert^2 \right)]+ c \me \int_0^t \by(s) (1+\lvert \hsa \buna(s)\rvert^2) ds \\
 \le \frac12 \me \sup_{s \in [0,t]}[\by(s)]^2+ cT + c \me \int_0^t [\by(s)]^2 ds.
\end{align*}
Taking the latter estimate and \eqref{NOT-1} into \eqref{FORMER}
\begin{equation*}
 \begin{split}
  \me \sup_{s\in [0,t]} [\by]^2(s)+8\me \int_0^t \by(s)
  [\lvert \sA\buna(t)\rvert^2+\alpha^2 \lvert \sA^\frac32 \buna(t)\rvert^2]ds\\
  \le 2 [\by(0)]^2 +c T +c \me \int_0^t [\by(s)]^2 ds.
 \end{split}
\end{equation*}
Applying Gronwall's and \eqref{GLAM-13} imply that there exists $C>0$ such that
\begin{equation*}
 \begin{split}
  \me \sup_{s\in [0,t]} [\by]^2(s)+8\me \int_0^t \by(s)
  [\lvert \sA\buna(t)\rvert^2+\alpha^2 \lvert \sA^\frac32 \buna(t)\rvert^2]ds\le K_0,
 \end{split}
\end{equation*}
where $$K_0:= [2(\vert \hsa \bu_0\rvert^2+\vert \sA \bu_0\vert^2)^2+CT]\cdot [1+C e^{KT}].  $$
Recalling the definition of $\by$ and by lower semicontinuity we infer from the last estimate that as $n\rightarrow \infty$
\begin{align}
 %\begin{split}
  \me \sup_{s\in [0,T]}[\lvert \hsa \bua(t)\rvert^2 +\alpha^2 \lvert \sA\bua(t)\rvert^2]^2\le K_0,\label{GLAM-14}\\
   8
  \me \int_0^T \Big(\lvert \hsa \bua(t)\rvert^2 +\alpha^2
  \lvert \sA\bua(t)\rvert^2\Big)\Big(\lvert \sA\bua(t)\rvert^2+\alpha^2 \lvert \sA^\frac32 \bua(t)\rvert^2\Big) ds\le K_0.\label{GLAM-15}
 %\end{split}
\end{align}
By straightforward calculations we easily derive from \eqref{GLAM-14} and \eqref{GLAM-15} the set of estimates
 \eqref{AP-EST-LAM-1}-\eqref{AP-EST-LAM-5} stated in
Proposition \ref{PROP-EST-LAM}.
\end{proof}
\section{Rate of convergence of the sequence $\bua$ to $\bu$}\label{RATE}
In this section we consider a sequence $\{\alpha_n; n \in
\mathbb{N}\}\subset (0,1)$ such that $\alpha_n \rightarrow 0$ as
$n \toup \infty$. For each $n\in \mathbb{N}$ let $\bu^{\alpha_n}$
be the unique solution to \eqref{LAM} and for each $R\in
\mathbb{R}$ define a family of stopping times $\tau_R^n$ by
\begin{equation}\label{STOP}
 \tau_R^n:=\inf\{ t\in [0,T]; \int_0^t \Lve \bu^{\alpha_n}(s)\Rve^2 ds \ge R\}.
\end{equation}
 Let $\bu$ be the solution of
the stochastic Navier-Stokes equations; that is, $\bu$ solves \eqref{NSE}.
In the following theorem we will show that by localization
procedure the sequence $\bu^{\alpha}$ converges strongly in
$L^2(\Omega,L^\infty(0,T;\h))$ and $L^2(\Omega,L^2(0,T;\ve))$ and
the strong speed of convergence is of order $O(\alpha)$. 
\begin{thm}\label{ORDER}
Let Assumption \ref{HYP-Q} holds and let $\bu_0$ be a
$\mathcal{F}_0$-measurable random variable such that $\me(\vert
\hsa \bu_0\rvert^2+\vert \sA \bu_0\vert^2)^2<\infty.$ Then there
exists $C>0$, $\kappa_0 (T)>0$  such that for any $R>0$ and $n\in
\mathbb{N}$ we have
\begin{equation}\label{ORDER-1}
 \me \sup_{s\in [0,t\wedge \tau_R^n]} \rvert \bu-\bu^{\alpha_n} \lvert^2 +4\me \int_0^{t\wedge \tau_R^n}\lvert \hsa [ \bu-\bu^{\alpha_n} ] \rvert^2 ds \le
 \alpha_n^2 \beta(R) \kappa_0 e^{C(T)\beta(R)T},
\end{equation}
where $C(T):= C(1+T)$, $\beta(R):=1+CRe^{CR}$ and $$
\kappa_0(T):=CT+CT^2+CK_0+CTK_1+C(1+T).$$
\end{thm}
\begin{proof}
Let us fix $n\in \mathbb{N}$ and let $\bu^{\alpha_n}$ be the
unique solution to \eqref{LAM}. Let $\bu$ be the unique solution
to \eqref{NSE}.
 Let us also fix $R>0$ and let $\tau^n_R$ be the stopping time defined above. For sake of simplicity we set
$\tau=t\wedge \tau^n_R$ for any $t\in [0,T]$ and
$\alpha:=\alpha_n$ for any $n \in \mathbb{N}$. Let
$\delta=\bu-\bua$. The stochastic process $\delta(t)$ with initial
condition $\delta(0)=0$  solves
\begin{equation*}
 \begin{split}
  d\delta(t)=[-\sA \delta(t)-B(\bu(t),\bu(t)+\na B(\bua(t),\bva(t))]dt= Q{(\bu(t))} -\na Q{(\bua(t))} )dW(t).\\
  %\delta(0)=0.
 \end{split}
\end{equation*}
Equivalently,
\begin{equation*}
 \begin{split}
  d\delta(t)+[\sA \delta(t)+B(\bu(t),\bu(t)-B(\bua(t),\bua(t))]dt-(Q{(\bu(t))} -\na Q{(\bua(t))} )dW(t)\\
  =[\na B(\bua(t),\bva(t))-B(\bua(t),\bua(t))]dt.\\
  %\%delta(0)=0.
 \end{split}
\end{equation*}
From It\^o's formula we infer that
\begin{align}
 %\begin{split}
  \sup_{s\in [0,\tau]}\rvert \delta(s)\lvert^2 +2\int_0^{\tau} \Lve \delta(s)\Rve^2 ds\le &
  2 \int_0^{\tau}\langle B(\bua(s),\bua(s)-B(\bu(s),\bu(s)), \delta(s)\rangle ds \nonumber \\
  & +2 \int_0^{\tau} \langle [\na B(\bua(s),\bva(s))-B(\bua(s),\bua(s))], \delta(s)\rangle ds\nonumber \\
  &+\int_0^{\tau} \Lve Q{(\bu(s))} -\na Q{(\bu(s))}  \Rve^2_{\mathcal{L}_2(\rK,\h)}ds\nonumber \\
  & +2 \int_0^{\tau}\langle \delta(s), (Q{(\bu(t))} -\na Q{(\bu(t))} )dW(s)\rangle \nonumber \\
  & \le \lvert J_1\rvert+\lvert J_2\rvert +\rvert J_3\rvert+  J_4  + J_5(t),\label{J0}
 %\end{split}
\end{align}
where
\begin{align*}
 J_1:=&2 \int_0^{\tau}\langle B(\bua(s),\bua(s)-B(\bu(s),\bu(s)), \delta(s)\rangle ds,\\
 J_2:=&2 \int_0^{\tau} \langle \na[B(\bua(s), \bva(s))-B(\bua(s), \bua(s))], \delta(s)\rangle ds,\\
 J_3:=&\int_0^{\tau} \langle (\na-\Id)B(\bua(s),\bua(s)),\delta(s)\rangle ds,\\
 J_4:=&\int_0^{\tau} \Lve Q{(\bu(s))} -\na Q{(\bu(s))}  \Rve^2_{\mathcal{L}_2(\rK,\h)}ds,\\
 J_5(t):=& 2 \int_0^{\tau}\langle \delta(s), (Q{(\bu(t))} -\na Q{(\bu(t))} )dW(s)\rangle.
\end{align*}
{Using the well-known fact $$\langle B(\bua(s),\bua(s)-B(\bu(s),\bu(s)), \delta(s) \rangle =-\langle B(\delta(s),\delta(s)), \bua(s)\rangle, $$
 the Cauchy-Schwarz inequality, the Gagliardo-Nirenberg inequality and the Young inequality we obtain the chain of inequalities
 \begin{align}
  \lvert J_1\rvert\le & 2 c \int_0^{\tau} \Lve \delta(s)\Rve_{\el^4 }\Lve \delta(s)\Rve \Lve \bua(s)\Rve_{\el^4} ds \nonumber\\
  \le & 2c  \int_0^{\tau} \lvert \delta(s) \rvert^\frac12 \Lve \delta(s)\Rve^\frac32 \Lve \bua(s)\Rve_{\el^4} ds,\nonumber\\
  \lvert J_1\rvert\le &  \frac12 \int_0^{\tau} \Lve \delta(s)\Rve^2 ds + c \int_0^{\tau} \lvert \delta(s)\rvert^2
  \Lve \bua(s)\Rve^4_{\el^4} ds.\label{J1}
 \end{align}}
%  \blue{Perhaps we can modify the red part by as follows:
%  Using the well-known fact
% $$\langle B(\bua(s),\bua(s)-B(\bu(s),\bu(s)), \delta(s) \rangle =-\langle B(\delta(s),\delta(s)), \bua(s)\rangle, $$
%  \begin{align}
% %   \lvert J_1\rvert\le & 2 c \int_0^{\tau} \Lve \delta(s)\Rve_{\el^4 }\Lve \delta(s)\Rve \Lve \bua(s)\Rve_{\el^4} ds \nonumber\\
% %   \le & 2c  \int_0^{\tau} \lvert \delta(s) \rvert^\frac12 \Lve \delta(s)\Rve^\frac32 \Lve \bua(s)\Rve_{\el^4} ds,\nonumber\\
%   \lvert J_1\rvert\le &  \frac12 \int_0^{\tau} \Lve \delta(s)\Rve^2 ds + c \int_0^{\tau} \lvert \delta(s)\rvert^2
%   \Lve \bua(s)\Rve^4_{\el^4} ds.\label{J1-b}
%  \end{align}
%  This way everything depends only on the approximating sequence $\bua$ (This is a kind of stability).
%  }

By using the definition $\bva=\bua +\alpha^2 \sA\bua$ we see that
\begin{align*}
\langle \na[B(\bua(s), \bva(s))-B(\bua(s), \bua(s))], \delta(s)\rangle=\alpha^2 \langle \na B(\bua(s), \sA\bua(s)), \delta(s)\rangle\\
= \alpha^2 \langle  B(\bua(s), \sA\bua(s)), \na \delta(s)\rangle\\
= \alpha^2 \langle  B(\bua(s), \na \delta (s)), \sA\bua(s)\rangle.
\end{align*}
From the last line along with the Cauchy-Schwarz inequality, the embedding
$\bh^2\subset \el^\infty$, \eqref{NASQ},  and \eqref{Poincare-2} it follows that
\begin{align}
 \lvert \langle \na[B(\bua(s), \bva(s))-B(\bua(s), \bua(s))], \delta(s)\rangle \rvert\le
 c \alpha^2 \Lve \bua(s) \Rve_{\el^\infty (\MO)} \lvert \na \hsa \delta(s) \rvert \lvert \sA \bua(s)\rvert,\nonumber \\
 \le c \alpha^2 \lvert \sA \bua(s)\rvert^2 \lvert  \hsa \delta(s) \rvert.\nonumber
\end{align}
Applying the Cauchy inequality in the last estimate implies that
\begin{equation*}
 \lvert \langle \na[B(\bua(s), \bva(s))-B(\bua(s), \bua(s))], \delta(s)\rangle \rvert\le \frac12 \Lve \delta(s) \Rve^2
 +c \alpha^4 \lvert \sA \bua(s)\rvert^4.
\end{equation*}
Thus,
\begin{equation}\label{J2}
 \lvert J_2\rvert\le \frac12 \int_0^{\tau} \Lve \delta(s)\Rve^2 ds+c \alpha^2 \int_0^{\tau} \alpha^2 \lvert \sA \bua(s)\rvert^2 ds.
\end{equation}
Invoking \cite[Lemma 4.1]{Titi} we infer that
\begin{align*}
 \lvert \langle (\na-\Id)B(\bua(s),\bua(s)),\delta(s)\rangle \rvert\le c \frac\alpha 2 \lvert B(\bua(s),\bua(s))\rvert
 \Lve \delta(s)\Rve,
\end{align*}
from which along Cauchy-Schwarz, the embedding $\bh^1(\MO)\subset \el^4(\MO)$, \eqref{Poincare-2} and
the Cauchy inequality  we derive that
\begin{align*}
 \lvert \langle (\na-\Id)B(\bua(s),\bua(s)),\delta(s)\rangle \rvert\le c \frac\alpha 2 \Lve \bua(s)\Rve_{\el^4(\MO)}
 \Lve \nabla \bua(s)\Rve_{\el^4(\MO)} \Lve \delta(s)\Rve\\
 \le c \frac\alpha 2 \Lve \bua(s)\Rve
 \vert \sA \bua(s)\vert \Lve \delta(s)\Rve\\
 \le \frac12 \Lve \delta(s)\Rve+c \frac{\alpha^2}4 \Lve \bua(s)\Rve^2 \lvert \sA\bua(s)\rvert^2.
\end{align*}
Hence
\begin{equation}\label{J3}
 \lvert J_3\rvert\le \frac12 \int_0^{\tau} \Lve \delta(s)\Rve^2 ds+ c \frac{\alpha^2}4 \int_0^{\tau} \Lve \bua(s)\Rve^2
 \lvert \sA \bua(s)\rvert^2 ds.
\end{equation}
Since $$ Q(\bu)-\na Q(\bua)= [Q(\bu)-\na Q(\bu)]+[\na Q(\bu)-\na Q(\bua)],$$ we infer that
\begin{align*}
 J_4\le &c \int_0^{\tau} \Lve Q{(\bu(s))} -\na Q{(\bu(s))}\Rve^2_{\mathcal{L}_2(\rK,\h)}ds+c \int_0^{\tau} \Lve [\na Q(\bu(s))-\na Q(\bua(s))] \Rve^2_{\mathcal{L}_2(\rK,\h)}\\ & \quad :=J_{4,1}+J_{4,2}.
\end{align*}
Since $Q{(\bu(t))} -\na Q{(\bu(t))} =\alpha^2 \sA \na Q{(\bu(t))}  $ we easily cheack that
\begin{align}
 J_{4,1} \le & c \alpha^2 \int_0^{\tau} \Lve \alpha \sA \na Q{(\bu(s))} \Lve^2_{\mathcal{L}_2(\rK,\h)} ds\nonumber\\
 \le & c \alpha^2 \int_0^{\tau} \Lve \alpha \sA^\frac12 \na \hsa Q{(\bu(s))}   \Lve^2_{\mathcal{L}_2(\rK,\h)}ds\nonumber \\
 \le & c \alpha^2 \Lve  \alpha \sA^\frac12 \na \Rve^2_{\mathcal{L}(\h)} \int_0^{\tau} \Lve \hsa Q{(\bu(s))}   \Lve^2_{\mathcal{L}_2(\rK,\h)}ds.\nonumber
\end{align}
Owing to Assumption \ref{HYP-Q}-\eqref{HYP-Q-ii} altogether with Remark \ref{Rem-HYP}-\eqref{Rem-HYP-ii} we obtain that
\begin{align*}
 J_{4,1}\le & c \alpha^2 \Lve  \alpha \sA^\frac12 \na \Rve^2_{\mathcal{L}(\h)} \int_0^{\tau} (1+\lvert \hsa \bu(s)\rvert)^2 ds\\
 \le & c \alpha^2 \Lve  \alpha \sA^\frac12 \na \Rve^2_{\mathcal{L}(\h)}[cT+ c \int_0^{\tau} \lvert \hsa \bu(s)\rvert^2 ds]
\end{align*}
It follows from the last estimate and \cite[Proof of Lemma 4.1]{Titi} that
\begin{equation}\label{J41}
 J_{4,1} \le c \frac{\alpha^2}{4}(T+ \int_0^{\tau} \lvert \hsa \bu(s)\rvert^2 ds).
\end{equation}
Now to estimate $J_{4,2}$ we use the fact that $\Lve \na \Rve_{\mathcal{L}(\h)}\le 1$ and Assumption \ref{HYP-Q}-\eqref{HYP-Q-i} and derive that
\begin{align}
 J_{4,2}\le & c \int_0^{\tau} \Lve \na \Rve^2_{\mathcal{L}(\h)}\Lve Q(\bu(s))-Q(\bua(s))\Rve^2_{\mathcal{L}_2(\rK,\h)}ds\nonumber \\
 \le & c \int_0^{\tau} \lvert \delta(s)\rvert^2 ds.\label{J42}
\end{align}
Thus, the two estimates \eqref{J41} and \eqref{J42} yield that
\begin{equation}\label{J4}
 J_4 \le \frac{\alpha^2}{4}\left(c T + c\int_0^{\tau} \lvert \hsa \bu(s)\rvert^2 ds\right)+c \int_0^{\tau} \lvert \delta(s)\rvert^2 ds.
\end{equation}
It follows from \eqref{J1}, \eqref{J2}, \eqref{J3} and \eqref{J4} that
\begin{equation}\label{J0-b}
 \sup_{s\in [0,\tau]}\lvert \delta(s)\rvert^2 + 4 \int_0^{\tau} \lvert \hsa \delta(s)\rvert^2 ds\le \alpha^2 \int_0^{\tau} \mathfrak{X}(s) \lvert \delta(s)\rvert^2 ds +\alpha ^2 \mathfrak{Y}(t) + \mathfrak{Z}(t),
\end{equation}
where
\begin{align*}
 \mathfrak{Y}(t):= &c \int_0^{\tau} \left( \alpha^2 \lvert \sA \bua(s) \rvert^4+ \lvert \hsa \bua(s) \rvert^2 \lvert \sA\bua(s)\rvert^2 + \lvert \hsa \bu(s)\rvert^2 \right)ds +cT,\\
 \mathfrak{Z}(t):= &\sup_{s\in [0,\tau]}\lvert J_5(s)\rvert+ c\int_0^{\tau} \lvert \delta(s)\rvert^2 ds,\\
 \mathfrak{X}(t):=& \Lve \bua(t) \Rve^4_{\el^4}.
 \end{align*}
 Note that using the definition of $\tau^\alpha_R$ it is not diffuclt to see that $$ \int_0^{\tau} \mathfrak{X}(s) ds \le R.$$
Note also that it follows from the Gronwall lemma  that 
\begin{equation*}
  \sup_{s\in [0,\tau]}\lvert \delta(s)\rvert^2 + 4 \int_0^{\tau} \lvert \hsa \delta(s)\rvert^2 ds\le (\alpha^2 \mathfrak{Y}(t)+\mathfrak{Z}(t))\cdot(1+CR e^ {CR} )
 \end{equation*}
Hence taking the mathematical expectation in  \eqref{J0-b} we obtain that
\begin{equation}\label{J0-c}
 \me \sup_{s\in [0,\tau]}\lvert \delta(s)\rvert^2 + 4 \me \int_0^{\tau} \lvert \hsa \delta(s)\rvert^2 ds\le (\alpha^2 \me \mathfrak{Y}(t)+\me \mathfrak{Z}(t))\cdot(1+CR e^ {CR} )
\end{equation}
Owing to the definition of $\mathfrak{Y}$ and Proposition
\ref{PROP-EST-LAM} we derive that
\begin{equation}\label{J0-d}
 \me \mathfrak{Y}(t)\le C T + C K_0.
\end{equation}
%Note that thanks the estimate \eqref{WAP-EST-LAM-6} stated in Corollary  \ref{COR-WAP} we have
%\end{equatio}
Now we deal with the estimation of $\me \mathfrak{Z}(t)$.
By the Burkholder-Davis-Gundy inequality we deduce that
\begin{align*}
 \me \sup_{s\in[0, t]} \lvert J_5(t)\rvert\le c T^\frac12 \me \sup_{s\in [0,\tau]}\lvert \delta(s) \rvert \sqrt{J_4}\\
 \le \frac12 \me \sup_{s\in [0,\tau]}\lvert \delta(s) \rvert^2 +cT J_4.
\end{align*}
Thus, by the last estimate, the inequality \eqref{J4} and the defintion of $\mathfrak{Z}(t)$ we derive that
\begin{equation}\label{J5}
 \me \mathfrak{Z}(t) \le \frac12 \me \sup_{s\in [0,\tau]}\lvert \delta(s) \rvert^2+\frac{\alpha^2}{4}\left(c T^2 + cT \int_0^{\tau} \lvert \hsa \bu(s)\rvert^2 ds\right)+c(1+T) \int_0^{\tau} \lvert \delta(s)\rvert^2 ds.
\end{equation}
Therefore, we derive from \eqref{NSE}, \eqref{J0-c}, \eqref{J0-d} and \eqref{J5} that
\begin{equation}\label{SPD1}
\begin{split}
 \me \sup_{s\in [0,\tau]} \rvert \delta(s)\lvert^2 +8 \me \int_0^{\tau}  \Lve \delta(s)\Rve^2 ds \le \alpha^2 \kappa_0 \beta(R) +\beta(R)C(T) \me \int_0^{\tau}\lvert \delta(s)\rvert^2 ds,
\end{split}
\end{equation}
where $C(T):= C(1+T)$, $\beta(R):=1+CRe^{CR}$ and $$
\kappa_0:=CT+CT^2+CK_0+CTK_1+C(1+T)$$ Applying Gronwall inequality
into \eqref{SPD1} implies that
\begin{equation}
 \me \sup_{s\in [0,\tau]} \rvert \delta(s)\lvert^2 +4\me \int_0^{\tau}  \Lve \delta(s)\Rve^2 ds \le \alpha_n^2 \beta(R) \kappa_0 e^{C(T)\beta(R)T},
\end{equation}
where the positive constant $\beta(R) \kappa_0 e^{C(T)\beta(R)T}$ does not depend on $n$ and the sequence $\alpha_n$.
\end{proof}
For every $R>0$,  $t\in [0,T]$ and any integer $n\geq 1$, let
\begin{equation}\label{Omega_M}
\Omega_{R}^n(t):=\left\{\omega\in\Omega\; : \; \int_{0}^{t}\Vert
\bu^{\alpha_n}(s,\omega)\Vert_{\el^4}^{4} ds\leq R\right\}.
\end{equation}
This definition shows that  $\Omega_R^n(t)\subset \Omega_R^n(s)$
for $s\leq t$ and that $\Omega_R^n(t)\in {\mathcal F}_t$ for any
$t\in [0,T]$.  Let
$\tau^n_R$ be the stopping time defined in \eqref{STOP}. It is not
difficult to show that $\tau_R^n=T$ on the set $\Omega_R^n(T)$.

Owing to the estimate we derive in the course of the proof of Theorem  \ref{ORDER} we derive the following result which tells us about the rate of convergence in probability of 
$\bua$ to $\bu$.
\begin{thm}\label{ORDER-PR}
Let the assumptions of Theorem \ref{ORDER} be satisfied. For any
integer $n\geq 1$ let $\eps_n(T)$  denote the error term defined
by
\[ \eps_n(T)=\sup_{s\in [0,T]} \vert \bu^{\alpha_n}(s) - \bu(s)\vert + \Big( \int_0^T \vert \hsa[\bu^{\alpha_n}(s) - \bu(s)] \vert ^2ds \Big)^{1/2} .\] Then $\eps_n(T)$ converges to 0 in probability
and the convergence is of order $O(\alpha_n)$. To be precise, for
any sequence $\big(\Gamma_n\big)_{n=1}^\infty $ converging to
$\infty$,
\begin{equation} \label{speed-proba-0}
\lim_{n \to \infty} \mathbb{P}\Big( \eps_n(T)\geq
\frac{\Gamma_n}{\alpha_n}\Big) = 0.
\end{equation}
Therefore, the sequence $\bu^\alpha$ converges to $u$ in
probability in $\h$ and the rate of convergence is of order
$O(\alpha)$.
\end{thm}
\begin{proof}
First from the definition of $\Omega_R^n(T)$ and Corollary \ref{COR-WAP} we can show that
 \begin{equation}\label{OMEGA}
\lim_{R\rightarrow \infty} \mathbb{P}\left(\Omega\backslash
\Omega_R^n(T)\right)=0.
\end{equation}
 In fact, it follows from the Markov inequality and \eqref{WAP-EST-LAM-6} that 
\begin{align*}
\sup_{n\geq 1} \mathbb{P}(\Omega\backslash \Omega_R^n(T))\le \frac1R \sup_{n\ge 1}\me \int_0^T \Vert \bu^{\alpha_n}(s) \Vert^4_{\el^4}ds \\
\le \frac1R (\me \lvert \bu_0+\sA\bu_0\rvert^4+CT )(1+Ce^{CT}) \to 0,
\end{align*}
 as $R\to \infty$.
% From definition of  and Markov inequality we derive
% that
% \begin{equation*}
% \mathbb{P}\left(\right)\le
% \frac{1}{R^2} \mathbb{E}\int_0^T \lVert \bua(t)\rVert^4_{\el^4}
% dt.
% \end{equation*}
% From this inequality and Corollary \ref{COR-WAP} we easily derive
% that
% \begin{equation}\label{OMEGA}
% \lim_{R\rightarrow \infty} \mathbb{P}\left(\Omega\backslash
% \Omega_R^n(T)\right)=0.
% \end{equation}
Now let $\{\Gamma_n; n\in \mathbb{N}\}$ be a sequence of positive
numbers such that $\Gamma_n\rightarrow \infty$ as $n\rightarrow
\infty$. By straightforward calculation we deuce that
\begin{align*}
\mathbb{P}\Big( \eps_n(T)\geq \Gamma_n \alpha_n \Big)\leq &\mathbb{P}\left(\Omega\backslash \Omega_R^n(T)\right) +
\mathbb{P}\left({\eps_n(T)\geq \Gamma_n \alpha_n },{\Omega_R^n(T)}\right)\\
\leq & \mathbb{P}\left(\Omega\backslash \Omega_R^n(T)\right) +
\mathbb{E}\left(1_{\Omega_R^n(T)} \eps_n(T)\geq \Gamma_n \alpha_n
\right).
\end{align*}
Using Markov inequality and \eqref{ORDER-1} in the last inequality
implies that
\begin{align*}
\mathbb{P}\Big( \eps_n(T)\geq \Gamma_n \alpha_n \Big) \leq &
\mathbb{P}\left(\Omega\backslash \Omega_R^n(T)\right) +
\frac1{\Gamma_n^2\alpha_n^2} \mathbb{E}\left(1_{\Omega_R^n(T)}
\eps^2_n(T)\right)\\
\le & \mathbb{P}\left(\Omega\backslash \Omega_R^n(T)\right) +
\frac1{\Gamma_n^2} \beta(R) \kappa_0 e^{C(T)\beta(R)T}
\end{align*}
where $C(T):= C(1+T)$, $\beta(R):=1+CRe^{CR}$ and $$
\kappa_0(T):=CT+CT^2+CK_0+CTK_1+C(1+T)$$ for
some constant $C>0$ independent of $n$. Note that $\beta(R)\le C
e^{2 CR}$ for any $R>0$. Hence there exist positive constants
$C>0$ and $C(T)$ such that for all $n\in \mathbb{N}$ we have
\begin{align}\label{CONV-PROB}
\mathbb{P}\Big( \eps_n(T)\geq \Gamma_n \alpha_n \Big) \le &
\mathbb{P}\left(\Omega\backslash \Omega_R^n(T)\right) +
\frac1{\Gamma_n^2} \kappa_0 e^{C(T)e^{CR} }.
\end{align}
Let $\{\Gamma_n; n\in \mathbb{N}\}$ be a sequence such that
$\Gamma_n\rightarrow \infty$ as $n\rightarrow \infty$ and
$$R(n)=\frac1C\log\left(\frac1{C(T)}\log\left(\log\left(\log(\Gamma_n)\right)\right)\right).$$
As $n\rightarrow \infty$ we see that $R(n)\rightarrow \infty$.
Thus, since there exists
 $c>0$ such that
 $\log\left(\log(\Gamma_n)\right)\le c \Gamma_n$ for $n$ large
 enough, it follows from \eqref{OMEGA} and \eqref{CONV-PROB}
 that
\begin{align*}
\mathbb{P}\Big( \eps_n(T)\geq \Gamma_n \alpha_n \Big) \le &
\mathbb{P}\left(\Omega\backslash \Omega_{R(n)}^n(T)\right) +
\frac1{\Gamma_n^2} \kappa_0 \log\left(\log(\Gamma_n)\right)
\rightarrow 0.
\end{align*}
as $n\to \infty$; this  concludes the proof.
\end{proof}
\section*{Acknowledgments}
P.~A.~Razafimandimby's research is funded by the FWF-Austrian Science Fund through the project M1487. Hakima Bessaih.~was supported in part by the Simons Foundation grant \#283308 and NSF grant \#1416689. 
The research on this paper was initiated during the visit of Razafimandimby at
 the University of Wyoming in November 2013 and was finished while P. R and H. B were visiting KAUST. They are both very grateful to both institutions for the warm and kind hospitality and great scientific atmosphere.

\end{document}